\documentclass[12pt]{article}
\usepackage{graphicx}
\usepackage{amssymb}
\usepackage{amsmath}
\usepackage{amsthm}

\usepackage{color}
\numberwithin{equation}{section}
\newtheorem{teo}{Theorem}[section]
\newtheorem{lem}{Lemma}[section]
\theoremstyle{definition}
\newtheorem{defi}{Definition}[section]
\newtheorem{remark}{Remark}[section]

\newcommand{\Swiech}{\'{S}wi\c{e}ch}
\newcommand{\R}{{\mathbb R}}  
\newcommand{\e}{\varepsilon}

\title{Entire solutions of fully nonlinear 
elliptic
 \\ equations with a 
superlinear gradient term }
\author{G.\,Galise\thanks{University of Salerno, Italy. Email: ggalise@unisa.it} , S. Koike\thanks{University of Tohoku, Japan. Email: koike@math.tohoku.ac.jp} , O. Ley\thanks{IRMAR - Institut de Recherche Math\'ematique de Rennes;\vskip0cm \hskip0.3cm INSA Rennes - Institut National des Sciences Appliqu\'ees - Rennes, France. \vskip0cm \hskip0.3cm Email: olivier.ley@insa-rennes.fr} \ and A.Vitolo\thanks{University of Salerno, Italy. Email: vitolo@unisa.it}}
\date{}  
\begin{document}
\maketitle

\begin{abstract}
In this paper we consider second order fully nonlinear operators with an additive superlinear gradient term. Like in the pioneering paper of Brezis for the semilinear case, we obtain the existence of entire viscosity solutions, defined in all the space, without assuming global bounds. A uniqueness result is also obtained for special gradient terms, subject to a convexity/concavity type assumption where superlinearity is essential and has to be handled in a different way from the linear case.

\end{abstract}

\section{Introduction}

We are interested in existence and uniqueness of solutions in $\mathbb R^n$ of fully nonlinear second order uniformly elliptic equations having superlinear growth in $u$ and $Du$. Solutions in the whole space are said to be \emph{entire}.\\
In the pioneering work \cite{Br}, Brezis considered the semilinear elliptic problem
\begin{equation}\label{eqBrezis}
\Delta u-|u|^{s-1}u=f(x),\quad s>1,
\end{equation}
showing that it is well-posed in ${\cal D}'(\mathbb R^n)$ without prescribing conditions at infinity for the data $f$ and $u$. The existence of a unique solutions $u\in L^s_{\rm loc}(\mathbb R^n)$ is proved assuming only $f\in L^1_{\rm loc}(\mathbb R^n)$. Moreover $u\geq0$ a.e. if $f\leq0$ a.e. in $\mathbb R^n$.\\
This result was extended by Esteban, Felmer and Quaas \cite{EFQ} for the larger class of fully nonlinear uniformly elliptic problems
\begin{equation}\label{eqEFQ}
F(D^2u)-|u|^{s-1}u=f(x)\quad \text{in $\mathbb R^n$}
\end{equation}
where $f\in L^n_{\rm loc}(\mathbb R^n)$ and the solution $u$ is intended in the $L^n$-viscosity sense.\\
In \cite{GV} Galise and Vitolo generalized the previous results for operators depending also on $x$ and on the gradient. Following the original ideas of Brezis, combined with viscosity type arguments, they proved in particular the existence of entire solutions of the uniformly elliptic equation \begin{equation}\label{eqGV}
F(x,D^2u)+H(Du)-|u|^{s-1}u=f(x),
\end{equation} where $F(x,\cdot)$ is merely a measurable functions, the Hamiltonian $H:\mathbb R^n\mapsto\mathbb R$ depends in a Lipschitz way on the gradient variable, $s$ is any real number strictly larger than 1 and $f\in L^n_{\rm loc}(\mathbb R^n)$.
Concerning the uniqueness it is a remarkable fact that if the principal part $F$ is independent on $x$, the well-posedness of (\ref{eqGV}) is ensured assuming only the continuity of the datum $f$, while in the general case further assumptions are needed in order to control the oscillation in the $x$-variable and the regularity of the solutions.\\
In a recent paper \cite{agq12}  Alarcon, Garcia Melian and Quaas proved various results of existence and uniqueness of distributional solutions of equation (\ref{eqBrezis}) in Sobolev spaces in the case of an additive gradient term with superlinear growth.

\noindent
Here we propose to study the well-posedness in the whole space for 
\begin{equation}\label{eq2}
F(x,D^2u)+ H(x,Du)-|u|^{s-1}u=f(x),
\end{equation}
where $H(\cdot,Du)$ may have a superlinear growth in the first derivative.\\
\noindent
In (\ref{eq2}), we will assume the following. The second order term
\begin{equation}\label{eq3}
\text{$F$ is $(\lambda,\Lambda)$-uniformly elliptic and  $F(x,0)=0$ a.e. in $\mathbb R^n$,}
\end{equation}
see~(\ref{eq5}) for a definition.
As regard $H:\mathbb R^n\times\mathbb R^n\mapsto\mathbb R,$ we will assume
\begin{equation}\label{eq3''}
H(x,0)=0\quad\text{a.e. in $\mathbb R^n$}
\end{equation}
  and there exist $\gamma_1,\gamma_m>0,$ $m> 1$ such that
\begin{equation}\label{eq3'}
\left|H(x,p)-H(x,q)\right|\leq\left(\gamma_1+\gamma_m(|p|^{m-1}+|q|^{m-1})\right)|p-q|
\end{equation}
for $p,\,q\in\mathbb R^n$ and a.e. $x\in\mathbb R^n$.\\
Note that it is always possible to assume~(\ref{eq3''}) and  $F(x,0)=0$ by replacing $f(x)$
with $f(x)-H(x,0)-F(x,0).$
Making these assumptions we have in mind as prototype the  equation
\begin{equation}\label{prototype}
{\cal P}^+_{\lambda,\Lambda}(D^2u)\pm c_1|Du|+ c_m|Du|^m-|u|^{s-1}u=f(x),
\end{equation}
where $c_1,\,c_m\in \mathbb R$ and
${\cal P}^+_{\lambda,\Lambda}$ is the Pucci extremal operator, see next Section for definitions.
Concerning the uniqueness part, we focus our attention on the case $c_m>0$ or $c_m<0$, referring to \cite{GV} 
for the case $ c_m=0$. Due to the assumptions, in particular (\ref{eq3-convexe}),  
it is worth  noticing 
 that the proof of the case $c_m >0$ is different from the case $c_m=0$, and the 
latter case, corresponding to (\ref{eqGV}) with a Lipschitz continuous Hamiltonian, cannot be obtained 
from our treatment of the case $c_m >0$ by continuity, as $c_m \to 0$.

Our existence result is the following

\begin{teo}\label{teoex}
Let $m\in[1,2]$ and $s>m.$
Suppose that~(\ref{eq3})-(\ref{eq3''})-(\ref{eq3'}) hold true. 
If $f\in L^n_{\rm loc}(\mathbb R^n),$  then the equation (\ref{eq2}) has an $L^n$-entire viscosity solution.
\end{teo}
When reinforcing~(\ref{eq3})-(\ref{eq3''})-(\ref{eq3'}), we are able to prove uniqueness
of the solutions. First of all we suppose that
\begin{equation}\label{continuityassumption}
x\to F(x,X)\;\text{is continuous for any $X\in{\cal S}^n$}.
\end{equation}
Moreover  we assume that for $R>0$ there exists a modulus of continuity $\omega_R$ such that 
\begin{eqnarray}\label{struct-F}
F(x,X)-F(y,Y)\leq\omega_R(|x-y|+\varepsilon^{-1}|x-y|^2)
\end{eqnarray}
whenever $|x|,\,|y|<R$ and $X,\,Y\in{\cal S}^n$ satisfy
\begin{eqnarray}\label{mat-ineq}
-\frac3\varepsilon \left(\begin{array}{cc}
I & 0\\
0 & I
\end{array}\right)\leq\left(\begin{array}{cc}
X & 0\\
0 & -Y
\end{array}\right)\leq\frac3\varepsilon \left(\begin{array}{cc}
I & -I\\
-I & I
\end{array}\right).
\end{eqnarray}
and there exist $\gamma_1,\gamma_m > 0$, $m>1$ and a modulus of continuity $\omega: [0,+\infty) \to [0,+\infty)$
such that
\begin{equation}\label{eq3-H}
\begin{split}
 |H(x,p+q)-H(y,p)|&\leq\, \omega(|x-y|)(|p|^m+1)\\
&+\,\left(\gamma_1+\gamma_m(|p|^{m-1}+|q|^{m-1})\right)|q|
\end{split}
\end{equation}
for $x,y,p,q\in\mathbb R^n$. 
Notice that~(\ref{eq3-H}) implies~(\ref{eq3'}), with a different choice of $\gamma_m$,
and combined with (\ref{eq3''}) yields $|H(x,p)|\leq \gamma (|p|^m+1)$ with $\gamma=\max(\gamma_1,\gamma_m)$.
We also assume
\begin{equation}\label{homogeneity}
F(x,\sigma X)=\sigma F(x,X)\quad
\text{for all $\sigma\in(0,1),$ $x\in\R^n,$ $X\in {\cal S}^n,$}
\end{equation}
and a convexity-type assumption on $H$: there exist $\underline{c},A>0$
and $\sigma_0\in (0,1)$
such that
\begin{equation}\label{eq3-convexe}
H(x,p)-\sigma H(x,\sigma^{-1}p)\leq (1-\sigma)(-\underline{c}|p|^m +A)
\quad\text{for $x,p\in\mathbb R^n,$ $\sigma\in (\sigma_0,1).$}
\end{equation}
This assumption covers the model case~(\ref{prototype})
when $H(x,p)=c_1(x)|p|+c_m(x)|p|^m$ with $m>1,$ $c_i(x)$ 
bounded uniformly continuous in $\R^n$
and $c_m(x)\geq \underline{c}>0$ for some $\underline{c}>0.$\\
We collect additional examples of 
Hamiltonians satisfying~(\ref{eq3-H})-(\ref{eq3-convexe}) in the subsection \ref{examples}.

\noindent The assumptions~(\ref{eq3-H})-(\ref{eq3-convexe}) will be used to deal
with the superlinear nonlinearity $H(x,Du)$ when performing
a kind of linearization of~(\ref{eq2}) through a technique borrowed by Barles-Koike-Ley-Topp \cite{bklt15} and Koike-Ley \cite{KL}, in the proof of the following
uniqueness theorem.

\begin{teo}\label{uniqueness} 
Let $m\in(1,2]$ and $s>m$. 
Suppose that~(\ref{eq3})-(\ref{continuityassumption})-(\ref{struct-F})-(\ref{eq3-H})-(\ref{homogeneity})-(\ref{eq3-convexe}) hold true. 
If $f$ is continuous with
\begin{equation}\label{eq2uniq}
\limsup_{|x|\to\infty}\frac{f^-(x)}{|x|^{\rho}}<\infty \quad
\text{for} \ \ 
\rho < \left\{\begin{array}{cl}
\frac{m(s-1)}{(m-1)s} & \hbox{\rm if} \ 1 < m \le \frac{2s}{s+1} \vspace{0.25cm}\\
\frac{2(s-m)}{s(m-1)} & \hbox{\rm if} \ \frac{2s}{s+1} < m\,,
\end{array}\right.
\end{equation}
then (\ref{eq2}) admits a unique entire continuous viscosity solution.\\
The same result holds if $-H$ satisfies~(\ref{eq3-convexe})
and~(\ref{eq2uniq}) holds with $f^+$ instead of $f^-.$
\end{teo}
In particular, we get uniqueness for the solutions of~(\ref{prototype})
when $c_m(x)$ is a  bounded uniformly continuous function which satisfies either $c_m(x)\geq \underline{c}>0$
(convex Hamiltonian) or $c_m(x)\leq -\underline{c}<0$
(concave Hamiltonian). As far as the growth is concerned, 
we have uniqueness, for instance if $m=2$ and $s>2$ 
provided~(\ref{eq2uniq}) holds for some $\rho <\frac{2(s-2)}{s}.$
The continuity of $f$ ensures that the two notions of $L^n$-viscosity solutions 
and classical viscosity solutions are equivalent, see \cite{KT}.
For further comments about the assumptions, see Section~\ref{sec:unicite}.

\subsection{Examples of Hamiltonians satisfying~(\ref{eq3-H})-(\ref{eq3-convexe})}\label{examples}

The Hamiltonian $H(x,p)=c(x)|p|^m+a(x)|p|^l$ 
satisfies~(\ref{eq3''})-(\ref{eq3-H})-(\ref{eq3-convexe})
if $c,a$ are bounded uniformly continuous in $\R^n,$ $c(x)\geq \underline{c}>0,$ 
$m>1$ and $0<l<m.$ 
To check the assumptions, we can take any $\sigma_0\in (0,1)$ and we 
use the inequalities $1-\sigma^{-s}\leq -s(1-\sigma)$ and 
$|\sigma -\sigma^s|\leq |1-s|(1-\sigma)$ for all
$s>0,$ $\sigma\in (0,1).$ For~(\ref{eq3-convexe}), we have
\begin{eqnarray*}
H(x,p)-\sigma H(x,\sigma^{-1}p) 
&=& c(x)(1-\sigma^{1-m})|p|^m +  a(x)(1-\sigma^{1-l})|p|^l\\
&\leq& (1-\sigma)
\left(-\underline{c}(m-1)|p|^m +\frac{||a||_\infty|l-1|}{\sigma_0^l}|p|^l\right)\\
&\leq&  (1-\sigma)\left(-\frac{\underline{c}(m-1)}{2}|p|^m +C\right),
\end{eqnarray*}
with $C=C(\underline{c},||a||_\infty,\sigma_0,m,l)$ by using Young inequality
\begin{eqnarray*}
\frac{||a||_\infty|l-1|}{\sigma_0^l} |p|^l\leq \frac{\underline{c}(m-1)}{2}|p|^m+C.
\end{eqnarray*}
To check~(\ref{eq3-H}), we use the inequality
$|p+q|^m-|p|^m\leq C(|p|^{m-1}+|q|^{m-1})|q|$ for $p,q\in\R^n,$
$m>0$ and $C=C(m).$ The constant $C$ below may vary line to line.
\begin{eqnarray*}
&& H(x,p+q)- H(y,p) \\
&=&
c(x)(|p+q|^m-|p|^m)+ a(x)(|p+q|^l-|p|^l)
+ (c(x)-c(y))|p|^m + (a(x)-a(y))|p|^l\\
&\leq&
C||c||_\infty(|p|^{m-1}+|q|^{m-1})|q|
+C||a||_\infty(|p|^{l-1}+|q|^{l-1})|q|
+\omega (|x-y|)(|p|^m+|p|^l)\\
&\leq&
C(|p|^{m-1}+|q|^{m-1}+1)|q|+ \omega (|x-y|)(|p|^m+1),
\end{eqnarray*}
by using Young inequality and where $\omega$ is a modulus of
continuity for $a,c.$

We may generalize the previous example by considering
$H(x,p)=\phi(x,|p|)|p|^m,$ with $m>1$ and
$\phi(\cdot, r)$ is bounded uniformly continuous in $\R^n$
uniformly with respect to $r\in \R_+,$
$\phi(x,r)\geq \phi_0 >0,$ $|\frac{\partial\phi}{\partial r}(x,r)|\leq C/r$
and $\phi(x,r)-\phi(x,\sigma^{-1}r)\leq C(1-\sigma)$ for all $\sigma\in (0,1).$
Details are left to the reader.

Notice that the previous example contains some nonconvex
Hamiltonians as $H(x,p)=c(x)\frac{(|p|^2-1)^2-1}{|p|^2+1}$ for instance.

Another class of examples is given by
\begin{eqnarray*}
H(x,p)=\mathop{\rm sup}_{\alpha\in A} \mathop{\rm inf}_{\beta\in B} \langle S_{\alpha\beta}(x)p,p\rangle^{m/2},
\end{eqnarray*}
where $m>1,$ $A,B$ are metric spaces, 
$S_{\alpha\beta}:\R^n\to\mathcal{S}^n$ is bounded uniformly
continuous uniformly with respect to $\alpha,\beta$
and there exists $\nu>0$ independent of $\alpha,\beta$ such
that $S_{\alpha\beta}(x)\geq \nu I.$
To prove~(\ref{eq3-convexe}), we notice that $H(x,p)\geq \nu^{m/2} |p|^m.$
Hence,
\begin{eqnarray*}
H(x,p)-\sigma H(x,\sigma^{-1}p)
= (1-\sigma^{1-m})H(x,p)
\leq -(1-\sigma)\nu^{m/2} |p|^m.
\end{eqnarray*}
For~(\ref{eq3-H}), we write
\begin{equation*}
\begin{split}
H(x,p+q)- H(y,p)
&\leq
\mathop{\rm sup}_{\alpha\in A,\beta\in B}
\{ \langle S_{\alpha\beta}(x)(p+q),p+q\rangle^{m/2}- \langle S_{\alpha\beta}(y)p,p\rangle^{m/2}\}\\
&=
\mathop{\rm sup}_{\alpha\in A,\beta\in B}
\{ \langle S_{\alpha\beta}(x)(p+q),p+q\rangle^{m/2}- \langle S_{\alpha\beta}(x)p,p\rangle^{m/2}\\
&+\langle S_{\alpha\beta}(x)p,p\rangle^{m/2}- \langle S_{\alpha\beta}(y)p,p\rangle^{m/2}\}\\
&\leq C\left(\left(|p|^{m-1}+|q|^{m-1}\right)|q|+\omega(|x-y|)|p|^{m}\right),
\end{split}
\end{equation*}
leading to~(\ref{eq3-H}) for a constant $C$ depending only on $m$ and $n$.

\section{Preliminaries}
We recall the definitions of  $L^p$-\emph{strong} and \emph{viscosity} solutions for second order elliptic equations
\begin{equation}\label{eq4}
G(x,u,Du,D^2u)=g(x)\quad\text{in $\Omega$},
\end{equation}
where $p\geq n$, $G:\Omega\times\mathbb R\times\mathbb R^n\times{\cal S}^n\mapsto\mathbb R$ and $g:\Omega\mapsto\mathbb R$ are measurable functions, $G$ is continuous in the last three variables, $\Omega$ is a domain (open connected set) and ${\cal S}^n$ denote the linear space of $n\times n$ real symmetric matrices equipped with the standard order:
$$X\leq Y\,\,\,\text{in ${\cal S}^n$}\Leftrightarrow \left\langle Xp,p\right\rangle\leq\left\langle Yp,p\right\rangle\,\,\, \forall p\in\mathbb R^n.$$
The identity matrix will be denoted by $I$ and the trace of $X\in{\cal S}^n$ with ${\rm Tr}(X)$.

We say that $G$  is $(\lambda,\Lambda)$-uniformly elliptic for $0<\lambda\leq\Lambda$ if 
\begin{equation}\label{eq5}
\lambda{\rm Tr}(Y)\leq G(x,r,p,X+Y)-G(x,r,p,X)\leq\Lambda{\rm Tr}(Y)\quad\text{a.e. $x\in\Omega$}
\end{equation}
for any $(r,p)\in\mathbb R\times\mathbb R^n$, $X,Y\in{\cal S}^n$ with $Y\geq0$, or equivalently
\begin{equation}\label{eq6}
{\cal P}^-_{\lambda,\Lambda}(Y-X)\leq G(x,r,p,Y)-G(x,r,p,X)\leq{\cal P}^+_{\lambda,\Lambda}(Y-X)\quad\text{a.e. $x\in\Omega$}
\end{equation}
for any $(r,p)\in\mathbb R\times\mathbb R^n$ and  $X,Y\in{\cal S}^n$. Here ${\cal P}^\pm_{\lambda,\Lambda}$ are the \emph{Pucci extremal operator} defined in the following way:
$${\cal P}^+_{\lambda,\Lambda}(X)=\sup_{\lambda I\leq A\leq\Lambda I}{\rm Tr}(AX),\quad {\cal P}^-_{\lambda,\Lambda}(X)=\inf_{\lambda I\leq A\leq\Lambda I}{\rm Tr}(AX).$$

\begin{defi}
\rm A function $u\in W^{2,p}_{\text{loc}}(\Omega)$
 is an $L^p$-strong subsolution, respectively supersolution, of (\ref{eq4}) if 
$$G(x,u(x),Du(x),D^2u(x))\geq g(x)\quad\text{a.e. in $\Omega$},$$
respectively
$$G(x,u(x),Du(x),D^2u(x))\leq g(x)\quad\text{a.e. in $\Omega$}\,.$$
We say that $u$ is an $L^p$-strong solution if it is both sub and supersolution.
\end{defi}

\begin{defi}\label{def1}
\rm A function $u\in C(\Omega)$ is called $L^p$-viscosity subsolution,  respectively supersolution, of (\ref{eq4}) provided for any $\varepsilon>0$, any open subset ${\cal O}\subset\Omega$ and any $\varphi\in W^{2,p}_{\text{loc}}(\Omega)$ such that 
$$
G(x,u(x),D\varphi(x),D^2\varphi(x))-g(x)\leq-\varepsilon\quad\text{a.e. in ${\cal O}$},
$$
respectively
$$
G(x,u(x),D\varphi(x),D^2\varphi(x))-g(x)\geq\varepsilon\quad\text{a.e. in ${\cal O}$},
$$
then $u-\varphi$ cannot have  a local maximum, respectively minimum, in ${\cal O}$.\\
Moreover if $u$ is both sub and supersolution then it is an $L^p$-viscosity solution.
\end{defi}

From definition \ref{def1} it is clear that $L^p$-viscosity notion imply the $L^q$ one for $q\geq p$ because of the inclusion $W^{2,q}_{\text{loc}}(\Omega)\subseteq W^{2,p}_{\text{loc}}(\Omega)$.\\
The test function $\varphi$ are continuous and twice differentiable a.e. \cite[Appendix C]{CCKS}. In the linear growth case the set of $L^p$-strong solution is a subset of the $L^p$-viscosity one. Conversely an $L^p$-viscosity solution that lies in $W^{2,p}_{\text{loc}}(\Omega)$ is an $L^p$-strong solution. For the proof and a general review of the $L^p$ theory of viscosity solution we refer to \cite{CCKS, CKSS}.
 In the superlinear case $L^p$-strong solutions continue to be  $L^p$-viscosity 
solution as stated in \cite[Theorem 3.1]{KS2}. 
Finally, we call continuous viscosity solutions, or $C$-viscosity solutions, the classical
notion of viscosity solutions (\cite{cil92, koike04}). 
When the data in the equations are continuous, they are equivalent
to $L^p$-viscosity solutions, see \cite{KT}.

  \begin{lem}\label{lemdiff}
Suppose that (\ref{eq3})-(\ref{eq3''})-(\ref{eq3'}) hold true. Let $u,v\in C(\Omega)$ be respectively $L^p$-viscosity sub and supersolution of 
\begin{equation}\label{eq12}
F(x,D^2u)+H(x,Du)-|u|^{s-1}u=f(x)
\end{equation}
and 
$$F(x,D^2v)+2^{m-1}\gamma_m|Dv|^m+\gamma_1|Dv|-|v|^{s-1}v=g(x)\,.$$
If $v\in W^{2,p}_{\rm loc}(\Omega)$ then the difference $w=u-v$ is an $L^p$-viscosity subsolution of the maximal equation
$${\cal P}^+_{\lambda,\Lambda}(D^2w)+2^{m-1}\gamma_m|Dw|^m+\gamma_1|Dw|=f(x)-g(x)$$
in $\left\{w>0\right\}$. 
\end{lem}
\begin{remark} As it will be clear from the proof, the result continues to hold if $m>2$.
\end{remark}
\begin{proof}
By contradiction assume that there exist $\varepsilon>0$, $\varphi\in W^{2,p}_{\rm loc}(\left\{w>0\right\})$, ${\cal O}\subset\left\{w>0\right\}$ open such that 
$${\cal P}^+_{\lambda,\Lambda}(D^2\varphi(x))+2^{m-1}\gamma_m|D\varphi(x)|^m+\gamma_1|D\varphi(x)|-(f(x)-g(x))\leq-\varepsilon\quad\text{a.e. in ${\cal O}$}$$
and $w-\varphi$ has a local maximum in ${\cal O}$. Thus $v+\varphi$ is a test function for $u$ and using the assumptions (\ref{eq3})-(\ref{eq3''})-(\ref{eq3'})
\begin{equation*}
\begin{split}
F(x,D^2(v+\varphi)(x))&+H(x,D(v+\varphi)(x))-|u(x)|^{s-1}u(x)-f(x)\\
&\leq F(x,D^2v(x))+{\cal P}^+_{\lambda,\Lambda}(D^2\varphi(x))+\gamma_m|D(v+\varphi)(x)|^m+\gamma_1|D(v+\varphi)(x)|\\
&-|u(x)|^{s-1}u(x)-f(x)\\
&\leq F(x,D^2v(x))+2^{m-1}\gamma_m|Dv(x)|^m+\gamma_1|Dv(x)|\\
&+{\cal P}^+_{\lambda,\Lambda}(D^2\varphi(x))+2^{m-1}\gamma_m|D\varphi(x)|^m+\gamma_1|D\varphi(x)|
-|u(x)|^{s-1}u(x)-f(x)\\
&\leq |v(x)|^{s-1}v(x)-|u(x)|^{s-1}u(x)-\varepsilon\\
&\leq-\varepsilon\quad\text{a.e. in ${\cal O}$}
\end{split}
\end{equation*}
a contradiction because $u$ is a subsolution of (\ref{eq12}).
\end{proof}
 
A fundamental tool we will use in the sequel  is the  ABP-estimate for solutions of uniformly elliptic equations. The classical ABP inequality states that in a bounded domain $\Omega$ 
$$\sup_\Omega u\leq\sup_{\partial\Omega}u+C{\rm diam}(\Omega)\left\|f^-\right\|_{L^n(\Omega)}$$
for any solution $u\in C(\Omega)$ of the maximal inequality ${\cal P}^+_{\lambda,\Lambda}(D^2u)+\gamma|Du|\geq f(x)$, where $C=C(n,\lambda,\Lambda,\gamma\,{\rm diam}(\Omega))$.
Such result has been extended in the  case $m>1$ of superlinear growth  in the gradient by Koike-\Swiech \cite{KS1}. 
In order to get the following ABP-estimates, deduced by \cite[Theorems 3.1-3.2]{KS1}, we also need the restriction $m\leq2$.
\begin{teo}\label{teoABP}
Let ${\rm diam}(\Omega)\leq1$ and let $u\in C(\overline\Omega)$ be an $L^p$-viscosity subsolution (resp. supersolution), $p\geq n$, of
$${\cal P}^+_{\lambda,\Lambda}(D^2u)+\gamma_1 |Du|+\gamma |Du|^m=f(x)\quad\text{in $\Omega$},$$
$$\left(resp.\,\,\,\,{\cal P}^-_{\lambda,\Lambda}(D^2u)-\gamma_1 |Du|-\gamma |Du|^m=f(x)\quad\text{in $\Omega$}\right),$$
with $\gamma_1$, $\gamma>0$, $m\in[1,2]$ and $f\in L^p(\Omega)$.\\ 
There exist two positive constants $$\hat{\delta}=\hat{\delta}(m,n,p,\gamma_1,\gamma,\lambda,\Lambda)<1,\, C=C(m,n,p,\gamma_1,\gamma,\lambda,\Lambda)$$ such that if
$${\rm diam}(\Omega)^{2-\frac np}\left\|f^-\right\|_{L^p(\Omega)}<\hat{\delta}$$
$$\left(resp. \,\,\,\,{\rm diam}(\Omega)^{2-\frac np}\left\|f^+\right\|_{L^p(\Omega)}<\hat{\delta}\right)$$
then
$$
\sup_{\Omega}u^+\leq\sup_{\partial \Omega}u^++C\,{\rm diam}(\Omega)^{2-\frac np}\left\|f^-\right\|_{L^p(\Omega)}
$$
$$
\hskip-1cm\left( resp.\,\,\,\,\sup_{\Omega}u^-\leq\sup_{\partial \Omega}u^-+C\,{\rm diam}(\Omega)^{2-\frac np}\left\|f^+\right\|_{L^p(\Omega)}\right),
$$
where $u^\pm=\max(\pm u,0)$.
\end{teo}

\begin{proof}
It is a straightforward consequence of \cite[Theorems 3.1-3.2]{KS1} by 
using the interpolation inequality 
$|Du|^m \leq (2-m)|Du|+(m-1)|Du|^2$ for $m\in [1,2].$
\end{proof}

\section{Uniform Estimates}
We denote by $B_r(x)$ the open ball centered at $x\in\mathbb R^n$ with radius $r>0$. When $x=0$ we write for simplicity $B_r$.\\
For $m\in[1,2]$ and $s>m$ we consider the Osserman's barrier function
\begin{equation}\label{eq7}
\phi_R(x)=\frac{C_RR^\mu}{\left(R^2-|x|^2\right)^\mu},\quad |x|<R
\end{equation}
where the positive constant $C_R$ is to be fixed and
\begin{equation}\label{eq8}
\mu=\left\{\begin{array}{cl} 
\displaystyle\frac{2}{s-1}& \text{if  \,$\displaystyle1 \le m \le \frac{2s}{s+1}$} \vspace{0.25cm}\\
\displaystyle\frac{m}{s-m}& \text{if  \,$\displaystyle\frac{2s}{s+1} < m < s$}.
\end{array}\right.
\end{equation}

\begin{lem}\label{lem1}
For any $\gamma_1,\gamma\geq0$ and $\delta>0,$
there exists $C_R>0$ such that the function $\phi_R$ defined in (\ref{eq7})
satisfies the differential inequality
$$
{\cal P}^+_{\lambda,\Lambda}(D^2\phi_R)+\gamma_1\left|D\phi_R\right|+\gamma\left|D\phi_R\right|^m-\delta\phi_R^s\leq0
$$
in $B_R$ in the classical sense.
For instance, we may choose
\begin{equation}\label{eq9}
C_R=\max\left\{a(1+\gamma_1R)^{\frac{1}{s-1}}R^{\mu-\frac{2}{s-1}},\ 
b\gamma^{\frac{1}{s-m}}R^{\mu-\frac{m}{s-m}}\right\}
\end{equation}
with 
$a^{s-1}=4\mu\delta^{-1}\max\left\{\Lambda(1+n+2\mu),1\right\}$ and  
$b^{s-m}=2^{m+1}\mu^m\delta^{-1}$. 
\end{lem}
\begin{proof}
Put $r:=|x|$ and $\phi_R(x)=\varphi(r)=C_RR^\mu\left(R^2-r^2\right)^{-\mu}$. The choice (\ref{eq8}) guarantees that
\begin{equation}\label{mu s >2}
\mu s=\max\left\{\mu+2,(\mu+1)m\right\}>2;
\end{equation}
in this way, since all the curvatures of $\phi_R$ are positive, a straightforward computation yields
\begin{equation*}
\begin{split}
{\cal P}^+_{\lambda,\Lambda}(D^2\phi_R(x))&+\gamma_1 \left|D\phi_R(x)\right|+\gamma\left|D\phi_R(x)\right|^m-\delta\phi_R^s(x)\\
&=\Lambda\left(\varphi''+\frac{n-1}{r}\varphi'\right)+\gamma_1\varphi'+\gamma\left(\varphi'\right)^m-\delta\varphi^s\\
&=\frac{C_RR^{\mu }}{\left(R^2-r^2\right)^{\mu s}}\left[2\Lambda\mu\left(R^2+(1+2\mu)r^2\right)\left(R^2-r^2\right)^{\mu s-\mu-2}\right. \\
&+2\mu\left(\Lambda(n-1)+\gamma_1r\right)\left(R^2-r^2\right)^{\mu s-\mu-1}\\
&+2^m\gamma\mu^mC_R^{m-1}R^{(m-1)\mu}r^m\left(R^2-r^2\right)^{\mu s-(\mu+1)m}
-\delta C_R^{s-1}R^{(s-1)\mu}\Big]\\
&\leq
\frac{C_RR^{\mu s}}{\left(R^2-r^2\right)^{\mu s}}\Big[2\Lambda\mu\left(1+n+2\mu\right)R^{\mu(s-1)-2}\\
&+2\gamma_1\mu R^{\mu(s-1)-1}+(2\mu)^m\gamma C_R^{m-1}R^{\mu(s-m)-m}-\delta C_R^{s-1}\Big].
\end{split}
\end{equation*}
Using (\ref{eq9}) we conclude
\begin{equation*}
\begin{split}
&\;2\Lambda\mu\left(1+n+2\mu\right)R^{\mu(s-1)-2}+2\gamma_1\mu R^{\mu(s-1)-1}+(2\mu)^m\gamma C_R^{m-1}R^{\mu(s-m)-m}-\delta C_R^{s-1}\\
\leq&\;\delta C_R^{s-1} \left(\frac{a^{s-1}(1+\gamma_1R)R^{\mu(s-1)-2}}{2C_R^{s-1}}+\frac{b^{s-m}\gamma R^{\mu(s-m)-m}}{2C_R^{s-m}}-1\right)\leq0.
\end{split}
\end{equation*}
\end{proof}

Following the same line of proof in 
\cite[Lemma 3.2]{GV} we prove the following uniform estimates result in 
 ``small'' balls.

\begin{lem}\label{lem2}
Suppose that (\ref{eq3})-(\ref{eq3''})-(\ref{eq3'}) hold true. Let $f\in L^n_{\rm loc}(\mathbb R^n)$ and $r$ be a positive number small enough such that Theorem \ref{teoABP} holds true in $B_{2r}$ with $p=n$ and $\gamma=2^{m-1}\gamma_m$. If $u\in C(\overline B_{2r})$ is an $L^n$-viscosity subsolution of 
\begin{equation*}
F(x,D^2u)+H(x,Du)-|u|^{s-1}u=f(x)\quad\text{in $B_{2r}$}
\end{equation*}
then
\begin{equation}\label{eq10}
\sup_{B_r}u\leq \frac{C_0}{r^\mu}+C\,r\left\|f^-\right\|_{L^n(B_{2r})},
\end{equation}
where $C_0=C_0\left(m,n,s,\gamma_1,\gamma_m,\Lambda\right)$, $C=C(m,n,\gamma_1,\gamma_m,\lambda,\Lambda)$ are positive constants.
\end{lem}

\begin{proof}
Since $\phi_{2r}(x)\to\infty$ as $|x|\to2r$, we can find $r<\overline r<2r$ such that $\phi_{2r}\geq u$ in $B_{2r}\backslash B_{\overline r}$ and $\left\{u>\phi_{2r}\right\}\subseteq B_{\overline r}$. By means of Lemma \ref{lem1}, setting $\gamma=2^{m-1}\gamma_m$, $\delta=1$ and $R=2r$, we construct the Osserman's barrier function $\phi_{2r}$, which is an $L^n$-strong supersolution of
$$F(D^2\phi_{2r})+\gamma_1\left|D\phi_{2r}\right|+2^{m-1}\gamma_m\, \left|D\phi_{2r}\right|^m-\phi_{2r}^s=0\quad\text{in $B_{2r}$}$$
and thus the difference $w=u-\phi_{2r}$ satisfies the inequality
$$
{\cal P}_{\lambda,\Lambda}^+(D^2w)+\gamma_1\left|Dw\right|+ 2^{m-1}\gamma_m\left|Dw\right|^m\geq f(x)\quad\text{in $\left\{w>0\right\}$
}
$$
in $L^n$-viscosity sense in view of Lemma \ref{lemdiff}. Using Theorem \ref{teoABP} we have
$$
u(x)\leq\phi_{2r}(x)+Cr\left\|f^-\right\|_{L^n(B_{2r})}
$$
from which (\ref{eq10}) follows.
\end{proof}

Reasoning as in Lemma \ref{lem2} on the function $v=-u$ it is easy to prove the next
\begin{lem}\label{lem3}
Suppose that (\ref{eq3})-(\ref{eq3''})-(\ref{eq3'}) hold true. Let $f$ and $r$ as in Lemma \ref{lem2}. If $u\in C(\overline B_{2r})$ is an $L^n$-viscosity solution of 
\begin{equation*}
F(x,D^2u)+H(x,Du)-|u|^{s-1}u=f(x)\quad\text{in $B_{2r}$}
\end{equation*}
then
\begin{equation}\label{eq11}
\sup_{B_r}|u|\leq \frac{C_0}{r^{\mu}}+C\,r\left\|f\right\|_{L^n(B_{2r})},
\end{equation}
with $C_0$ and $C$ as in (\ref{eq10}).
\end{lem}

\section{Existence}
In order to prove Theorem \ref{teoex} we will use the uniform bounds of Section 3.

\noindent
\emph{Proof of Theorem \ref{teoex}.} 
By \cite[Theorem 1 (i)]{Sir} we can solve any Dirichlet problem for the equation (\ref{eq2}) in the ball $B_k$, $k\in\mathbb N$, with continuous boundary condition. Choose a solution $u_k$ for any $k$. Let $\left\{B_{r}(x_i)\right\}_{i=1,\ldots,K}$ be a covering of $B_k$ such that
$$
B_k\subseteq\bigcup_{i=1}^KB_{r}(x_i)\subseteq\bigcup_{i=1}^KB_{2r}(x_i)\subseteq B_{k+1}
$$
and $r>0$, for $i=1,\ldots,K$, small enough as in Lemma \ref{lem3}. In this way for any $h>k$, using (\ref{eq11}), one has 
$$
\sup_{B_{r}(x_i)}|u_h|\leq \frac{C_0}{r^{\mu}}+C\,r\left\|f\right\|_{L^n(B_{2r}(x_i))}
$$
and 
$$
\sup_{B_k}|u_h|\leq\max_{i=1,\dots,K}\sup_{B_{r}(x_i)}|u_h|\leq C_1\left(1+\left\|f\right\|_{L^n(B_{k+1})}\right)
$$
with {$C_1=C_1(k,m,n,s,\gamma_1,\gamma_m,\lambda,\Lambda)$.} 
Using the $C^\alpha$-estimates \cite[Theorem 2]{Sir}
$$
\left\|u_h\right\|_{C^\alpha(B_k)}\leq C_2\left(1+\left\|f\right\|_{L^n(B_{k+1})}\right)
$$
for a positive constant $C_2$ independent of $h>k$. By a diagonal process we can extract a subsequence $u_{k_h}$ converging locally uniformly to a function $u\in C(\mathbb R^n)$. From the stability result of \cite[Theorem 4]{Sir} $u$ in an $L^n$-viscosity solution of (\ref{eq2}).
\hfill$\Box$

\section{Uniqueness}
\label{sec:unicite}

This Section is concerned with the uniqueness of $C$-viscosity entire solutions. As announced in the Introduction, we assume throughout that the Hamiltonian is actually superlinear, satisfying the convexity type assumption (\ref{eq3-convexe}), and refer to \cite{GV} for Lipschitz continuous Hamiltonians.\\
We start with a few remarks.

The condition $s\in (m,+\infty)$ in Theorem~\ref{teoex} is necessary in order to obtain a uniqueness result.
In fact  the functions
$$u(x_1,\ldots,x_i,\ldots,x_n)=\alpha\exp(\pm\sqrt{2}x_i)+1$$
are solutions in $\mathbb R^n$ of the equation 
$$\Delta u+\frac12|Du|^2-|u|u=-1$$
for any $\alpha\geq0$. Similarly the functions $v=-u$ satisfy $\Delta v-\frac12|Dv|^2-|v|v=1$ in $\mathbb R^n$.

Assumptions (\ref{eq3-H})-(\ref{eq3-convexe}) are needed to deal with the strong
superlinear nonlinearity in the Hamiltonian $H$ when
performing a kind of linearization, see Lemma~\ref{lem-lineariz} in the proof
of Theorem~\ref{uniqueness}.

The assumption~(\ref{eq2uniq}) gives a limiting growth on the data $f.$
Note that inequality
\begin{equation*}
\rho < \left\{\begin{array}{cl}
\frac{m(s-1)}{(m-1)s} & \hbox{\rm if} \ 1 < m \le \frac{2s}{s+1} \vspace{0.25cm}\\
\frac{2(s-m)}{s(m-1)} & \hbox{\rm if} \ \frac{2s}{s+1} < m\,,
\end{array}\right.
\end{equation*}
can be rewritten in the more synthetic way
\begin{equation}\label{val-rho}
\rho < \frac{2m'}{\mu s}
\end{equation}
where $\mu$ is introduced in (\ref{eq8}) and 
$m'$ is the conjugate of $m$
defined by  $\frac{1}{m}+\frac{1}{m'}=1$.

The following Lemma says how the growth of $u^\pm$ depends on the growth of 
$f^\mp$ at infinity.

\begin{lem}\label{lemuniq} 
Under the assumptions of Theorem~\ref{teoex},
let $\rho\geq0$ and assume
\begin{equation}\label{limsup:f-}
\limsup_{|x|\to\infty}\frac{f^-(x)}{|x|^{\rho}}=:l<\infty.
\end{equation}
Then any subsolution of (\ref{eq2}) satisfies 
$$
\limsup_{|x|\to\infty}\frac{(u^+)^s(x)}{ |x|^\frac{\mu s\rho}{2} } <\infty.
$$
where $\mu$ is defined in (\ref{eq8}).
The same result holds
replacing $f^-$ by $f^+,$ $u^+$ by $u^-$ and ``subsolution'' by ``supersolution''.

\end{lem}

\begin{proof}
Let $\varepsilon_0>0$ be such that 
$$2^{2+\rho}(l+1)\omega_n^{\frac1n}\varepsilon_0^2<\hat\delta<1$$ 
as required by Theorem \ref{teoABP}, where $\omega_n$ is the volume of the unit ball in $\mathbb R^n$.
Set $r_0=\varepsilon_0|x_0|^{-\frac{\rho}{2}}$. For $|x_0|$ big enough, by assumption (\ref{limsup:f-}) we get
\begin{equation*}
\begin{split}
2r_0\left\|f^-\right\|_{L^n(B_{r_0}(x_0))}&\leq\,2(l+1)\,r_0\left\||x|^{\rho}\right\|_{L^n(B_{r_0}(x_0))}\\
&\,\leq2^{\rho+1}(l+1)\omega_n^{\frac1n}\left(r_0^{\rho+2}+\varepsilon_0^2\right)\\
&\,\leq2^{\rho+2}(l+1)\,\omega_n^{\frac1n}\varepsilon_0^2<\hat \delta,
\end{split}
\end{equation*}
by our choice of $\varepsilon_0$. 
In this way Theorem \ref{teoABP} applies in $B_{r_0}{(x_0)}$ and  
Lemma \ref{lem2} yields, for $x_0$ far away from the origin, the estimate
$$u(x_0)\leq\sup_{B_{\frac{r_0}{2}}(x_0)}u\leq\frac{C_1}{r_0^\mu}+C$$
with  $\mu$ given by (\ref{eq8}), where $C_1=2^{\mu}C_0$ 
. In this way 
$$(u^+)^s(x_0)\leq2^{s-1}\left(\frac{C_1^s}{r_0^{\mu s}}+C^s\right)=2^{s-1}\left(\frac{C_1^s}{\varepsilon_0^{\mu s}}
|x_0|^\frac{\mu s \rho}{2}+C^s\right)$$
and $$
\limsup_{|x|\to\infty}\frac{(u^+)^s(x)}{|x|^\frac{\mu s \rho}{2}}\leq2^{s-1}\frac{C_1^s}{\varepsilon_0^{\mu s}}$$
as claimed. We prove the second part of the lemma. If $u$ is a supersolution
of~(\ref{eq2}), then 
the function $v=-u$ is a subsolution in $\mathbb R^n$ of 
$$-F(x,-D^2v)-H(x,-Dv)-|v|^{s-1}v=-f(x),$$
where the operators $-F(x,-X)$ and $-H(x,-p)$ turn out to satisfy ~(\ref{eq3})-(\ref{eq3''})-(\ref{eq3'}), and so the first part yields
$$\limsup_{|x|\to\infty}\frac{(u^-)^s(x)}{|x|^\frac{\mu s \rho}{2}}=\limsup_{|x|\to\infty}\frac{(v^+)^s(x)}{|x|^\frac{\mu s \rho}{2}}<\infty.$$
\end{proof}

We now turn to the proof of the uniqueness theorem. We will use the inequality
\begin{equation}\label{eqdelta}
|u|^{s-1}u-|v|^{s-1}v>\delta(s)(u-v)^s\quad\text{for $u>v$ and $s>1,$}
\end{equation}
where $\delta(s)$ is a positive constant.

\begin{proof}[Proof of Theorem~\ref{uniqueness}]
By contradiction let us assume that $u$ and $v$ are both viscosity solutions of (\ref{eq2})
such that
\begin{equation}\label{contrad}
\theta:=u(x_0)-v(x_0)>0,\quad x_0\in\mathbb R^n.
\end{equation}
The following lemma performs a kind of linearization of the equation.
\begin{lem}\label{lem-lineariz}
Assume that $F$ satisfies~(\ref{continuityassumption}), (\ref{struct-F}), (\ref{homogeneity}), (\ref{eq6})
and $H$ satsifies~(\ref{eq3-H}), (\ref{eq3-convexe}).
For any $\sigma\in(0,1),$ the function $w_\sigma :=u-v_\sigma:=u-\sigma v$ is a subsolution
in $\mathbb R^n$ of the extremal PDE
\begin{eqnarray}\label{eqdiff}
&& {\cal P}^+_{\lambda,\Lambda}(D^2w_\sigma)+\gamma_1|Dw_\sigma|+(1-\sigma)^{1-m}\,\tilde{\gamma} |Dw_\sigma|^m  
\nonumber\\
&& \hspace*{0.7cm }-\left(|u|^{s-1}u-|v_\sigma|^{s-1}v_\sigma\right)+(\sigma-\sigma^s)|v|^{s-1}v
\geq (1-\sigma)(f(x)-A),
\end{eqnarray}
where 
\begin{eqnarray}\label{tildgam}
\tilde{\gamma} = \gamma_m + \frac{(m-1)^{m-1}\gamma_m^m}{m^m \underline{c}^{m-1}} 
\end{eqnarray}
and $\gamma_1,\gamma_m, \underline{c},A$ are the constants appearing in~(\ref{eq3-H})-(\ref{eq3-convexe}).
\end{lem}
A proof of the lemma is provided in the Appendix.

For $\sigma$ close to 1 $$w_\sigma(x_0)>\frac\theta2.$$
Applying Lemma~\ref{lem1} with
$\gamma= (1-\sigma )^{1-m} \tilde{\gamma}$ and $\delta=\delta(s)$
given by (\ref{eqdelta}),
the function 
$$\phi_R(x)=\frac{C_RR^\mu}{(R^2-|x|^2)^\mu},\,\quad |x_0|<R,$$
is a solution in $B_R$ of
\begin{equation}\label{eqosserman}
{\cal P}^+_{\lambda,\Lambda}(D^2\phi_R) +
\gamma_1 \left|D \phi_R \right|+(1-\sigma )^{1-m} \tilde{\gamma}\left|D \phi_R \right|^m
-\delta(s) \phi _R^s\leq0
\end{equation}
for 
\begin{equation*}
C_R=\max\left\{a(1+\gamma_1 R)^{\frac{1}{s-1}}R^{\mu-\frac{2}{s-1}},\ 
b (1-\sigma)^{\frac{1-m}{s-m}}\tilde{\gamma}^{\frac{1}{s-m}}
R^{\mu-\frac{m}{s-m}}\right\}.
\end{equation*}
We set  $1-\sigma=KR^{-m'}$ for $K$ to be fixed. We have
$$\phi_R(x_0)=\frac{1}{\left(1-\frac{|x_0|^2}{R^2}\right)^\mu} R^{-\mu} C_R.$$
Noticing that
$$R^{-\mu}  R^{\frac{1}{s-1}} R^{\mu-\frac{2}{s-1}}=  R^{-\frac{1}{s-1}} \to 0
\text{ as $R\to +\infty$}$$
and
$$R^{-\mu}  (1-\sigma)^{\frac{1-m}{s-m}} R^{\mu-\frac{m}{s-m}}= K^{\frac{1-m}{s-m}},$$
we obtain
\begin{equation}\label{eqK}
\lim_{R\to+\infty}\phi_R(x_0)=b \tilde{\gamma}^{\frac{1}{s-m}} K^{\frac{1-m}{s-m}}
=\frac{\theta}{8}
\end{equation}
by fixing $K=\left(\frac{8b}{\theta}\right)^{\frac{s-m}{m-1}}\tilde{\gamma}^{\frac{1}{m-1}}.$
In this way for $R$ big enough
$$w_\sigma(x_0)>\frac{\theta}{2}>\phi_R(x_0)+\frac{\theta}{4}$$
and the difference $w_\sigma-\phi_R$ attains its maximum in $B_R$ at a point, say $x_R$,  such that 
$$(w_\sigma-\phi_R)(x_R)\geq(w_\sigma-\phi_R)(x_0)>\frac\theta4.$$
We deduce $w_\sigma(x_R)>\frac\theta4$ and 
$$\left(|u|^{s-1}u-|v_\sigma|^{s-1}v_\sigma\right)(x_R)>\delta(s) w_\sigma^s(x_R)>\delta(s)\left(\phi_R(x_R)+\frac\theta4\right)^s$$
because of (\ref{eqdelta}).\\
Using $\phi_R$ as test function for $w_\sigma$ at $x_R$
in (\ref{eqdiff}) and the inequality (\ref{eqosserman}) one has
\begin{equation}\label{eq3uniq}
\begin{split}
(1-\sigma)(f(x_R)-A)&\leq 
{\cal P}^+_{\lambda,\Lambda}(D^2\phi_R(x_R))
+ \gamma_1 |D \phi_R(x_R)|+ (1-\sigma)^{1-m}\tilde{\gamma} |D \phi_R(x_R)|^m\\
&\hskip-1cm  -\left(|u|^{s-1}u-|v_\sigma|^{s-1}v_\sigma\right)(x_R)+(\sigma-\sigma^s)|v|^{s-1}v(x_R)\\
&\hskip-1cm \le {\cal P}^+_{\lambda,\Lambda}(D^2\phi_R(x_R))+
\gamma_1 |D \phi_R(x_R)|+ (1-\sigma)^{1-m}\tilde{\gamma} |D \phi_R(x_R)|^m\\
&\hskip-1cm 
 -\delta(s)\left(\phi_R(x_R)+\frac\theta4\right)^s
+ (\sigma-\sigma^s)|v|^{s-1}v(x_R)\\
&\hskip-1cm \leq
-\delta(s) \left(\frac\theta4\right)^s+(\sigma-\sigma^s)|v|^{s-1}v(x_R).
\end{split}
\end{equation}
The function $f^-$ satisfies (\ref{eq2uniq}) with some $\rho <\frac{2m'}{\mu s} <m'$
(see (\ref{mu s >2}) and  (\ref{val-rho})). It follows
$$\limsup_{R\to+\infty}-(1-\sigma)(f(x_R)-A)
\leq K\limsup_{R\to+\infty}\frac{(f^-(x_R)+A)}{R^{m'}}=0.$$
On the other hand, using the elementary inequality $\sigma-\sigma^s\leq(s-1)(1-\sigma)$ for any $\sigma\in(0,1)$ and applying Lemma~\ref{lemuniq} with $u=v$ and $\rho <\frac{2m'}{\mu s},$ which implies
$\frac{\mu s\rho}{2}<m',$
we conclude  
$$\limsup_{R\to+\infty}(\sigma-\sigma^s)|v|^{s-1}v(x_R)\leq K(s-1)\limsup_{R\to+\infty}\frac{(v^+)^s(x_R)}{R^{m'}}=0$$ and (\ref{eq3uniq}) produces a contradiction for $R$ big enough.

We give a sketch of the proof of the case when $-H$ satisfies~(\ref{eq3-convexe})
and $f^+$ satisfies~(\ref{eq2uniq}).
Arguing as above by contradiction, assuming that two solutions $u,v$ satisfies~(\ref{contrad}),
 we now consider $w_\sigma:=u_\sigma-v:=\sigma u-v$ for $\sigma\in(0,1).$ 
Using~(\ref{eq3-H}) and the fact that $-H$ satisfies~(\ref{eq3-convexe})
one proves in a similar way as above that $w_\sigma$ is a subsolution
in $\mathbb R^n$  of 
\begin{eqnarray*}
&& {\cal P}^+_{\lambda,\Lambda}(D^2w_\sigma)
+\gamma_1 |Dw_\sigma|
+(1-\sigma)^{1-m}\,\tilde{\gamma} |Dw_\sigma|^m 
-\left(|u_\sigma|^{s-1}u_\sigma-|v|^{s-1}v\right)+(\sigma^s-\sigma)|u|^{s-1}u
\nonumber\\
&\geq &(\sigma -1)(f(x)+A),
\end{eqnarray*}
where $\tilde{\gamma}$ is still defined by~(\ref{tildgam}).
Setting as above $1-\sigma=KR^{-m'}$ with $K$ as in (\ref{eqK}),
denoting with $x_R$ a maximum point in $B_R$ of $w_\sigma-\phi_R$ and arguing as for (\ref{eq3uniq}), we obtain 
$$(\sigma -1)(f(x_R)+A)\leq-\delta\left(\frac\theta4\right)^s+(\sigma^s-\sigma)|u|^{s-1}u(x_R).$$
We obtain a contradiction as above using this time that $f^+$ satisfies (\ref{eq2uniq})
and applying the second part of Lemma~\ref{lemuniq} which gives a limiting growth for $u^+.$
\end{proof}

\begin{remark}\label{rmk1}
When $H$ satisfies
(\ref{eq3''})-(\ref{eq3'}), the subsolutions of (\ref{eq2})
are bounded from above by requiring the uniform bound of the local $L^n$-norm of $f^-$
$$\sup_{x\in\mathbb R^n}\left\|f^-\right\|_{L^n(B_1(x))}<\infty.$$
To see this it is sufficient to fix $r$ small enough in Lemma \ref{lem2} and using (\ref{eq10}). 
In this case Theorem \ref{uniqueness} holds true replacing (\ref{eq2uniq}) with the weaker assumptions
$$\limsup_{|x|\to\infty}\frac{f^-(x)}{|x|^\frac{\mu s\rho}{2}}<\infty$$ 
for $0\leq\rho< \frac{2m'}{\mu s}$.\\
Accordingly, when $-H$ satisfies~(\ref{eq3-convexe}), 
the supersolutions of (\ref{eq2}) are bounded from below 
if $\displaystyle \sup_{x\in\mathbb R^n}\left\|f^+\right\|_{L^n(B_1(x))}$ is finite
and Theorem \ref{uniqueness} continues to work under the assumption
$$\limsup_{|x|\to\infty}\frac{f^+(x)}{|x|^\frac{\mu s\rho}{2}}<\infty$$ for 
$0\leq\rho< \frac{2m'}{\mu s}$. 
\end{remark}

\appendix

\section{Appendix}

\subsection*{Proof of Lemma~\ref{lem-lineariz}}

The proof borrows arguments from \cite{bklt15, KL}, we provide it for
reader's convenience.

For $\sigma\in (0,1)$, from~(\ref{homogeneity}),
the function $v_\sigma:=\sigma v$ is a solution,
so a supersolution, of 
$$
F(x,D^2v_\sigma)
+\sigma H(x,\frac{Dv_\sigma}{\sigma})
-\sigma^{1-s}|v_\sigma|^{s-1}v_\sigma=\sigma f(x)\quad{\rm in}\;\;\mathbb R^n.
$$

We shall show that $w_\sigma=u-v_\sigma$ is a viscosity subsolution 
of the extremal PDE~(\ref{eqdiff}).
For $\phi\in C^2(\R^n)$, we suppose that $w_\sigma-\phi$ attains a local maximum 
at $\hat x\in\R^n$. 
We may suppose that $(w_\sigma-\phi)(\hat x)=0>(w_\sigma-\phi)(x)$ for $x\in B_r(\hat x)
\setminus \{ \hat x\}$ with a small $r\in (0,1)$.

Let $(x_\e, y_\e)\in B:=\overline{B}_r(\hat x)\times \overline{B}_r(\hat x)$ 
be a maximum point of $u (x)-v_\sigma (y)-(2\e)^{-1}|x-y|^2-\phi (x)$ over $B$. 
Since we may suppose $\lim_{\e\to 0}(x_\e,y_\e)=(\hat x,\hat x)$, and 
moreover 
$\lim_{\e\to 0}(u (x_\e),v_\sigma(y_\e))=(u (\hat x),v_\sigma(\hat x))$, 
it follows that $(x_\e,y_\e)\in {\rm int}(B)$ for small $\e.$ 
Hence, in view of Ishii's lemma (e.g., Theorem 3.2 in \cite{cil92}),
setting $p_\e=\e^{-1}(x_\e-y_\e)$, 
we find $X_\e ,Y_\e\in \mathcal{S}^n$ such that 
$(p_\e+D\phi (x_\e) ,X_\e+D^2\phi (x_\e))\in \overline J^{2,+}u (x_\e)$, $(p_\e,Y_\e)
\in \overline J^{2,-}v_\sigma (y_\e)$, and~(\ref{mat-ineq}) holds.
Thus, from the definition, we have
\begin{eqnarray}\label{ineg-sous}
F(x_\e, X_\e+D^2\phi (x_\e))
+H(x_\e, p_\e+D\phi (x_\e) )
-|u(x_\e)|^{s-1}u(x_\e)\geq f(x_\e)
\end{eqnarray}
and
\begin{eqnarray}\label{ineg-sur}
F(y_\e, Y_\e)
+\sigma H(y_\e , \frac{p_\e}{\sigma})
-\sigma^{1-s}|v_\sigma(y_\e)|^{s-1}v_\sigma(y_\e)\leq\sigma f(y_\e)
\end{eqnarray}

From~(\ref{struct-F}) and~(\ref{eq6}), we have
\begin{eqnarray*}
F(y_\e, Y_\e)\geq F(x_\e, X_\e)-\omega_r(|x_\e-y_\e|+\frac{|x_\e-y_\e|^2}{\e})
\end{eqnarray*}
and 
\begin{eqnarray*}
F(x_\e, X_\e+D^2\phi (x_\e))- F(x_\e, X_\e) \leq  {\cal P}^+_{\lambda,\Lambda}(D^2\phi (x_\e)).
\end{eqnarray*}

From~(\ref{eq3-H}) and~(\ref{eq3-convexe}), 
and choosing $\e$ small enough in order that $\underline{c}(1-\sigma)> \omega(|x_\e-y_\e|),$  
we get
\begin{eqnarray*}
&& H(x_\e, p_\e+D\phi (x_\e) ) - \sigma H(y_\e , \frac{p_\e}{\sigma})\\
& =& 
H(x_\e, p_\e+D\phi (x_\e) )- H(y_\e , p_\e) + H(y_\e , p_\e) - \sigma H(y_\e , \frac{p_\e}{\sigma})\\
&\leq&
\omega(|x_\e-y_\e|) (|p_\e|^m+1)+\left(\gamma_1+\gamma_m\left(|p_\e|^{m-1} +|D\phi (x_\e)|^{m-1}\right)\right)|D\phi (x_\e)|
- (1-\sigma)( \underline{c} |p_\e|^m -A)\\
&\leq&
\mathop{\rm sup}_{r\geq 0}\left\{ (- (1-\sigma)\underline{c} + \omega(|x_\e-y_\e|))r^{\frac{m}{m-1}}
+\gamma_m |D\phi (x_\e)| r \right\}\\
&& \hspace*{2cm}+\gamma_m |D\phi (x_\e)|^{m} +\gamma_1 |D\phi (x_\e)|
+ (1-\sigma)A + \omega(|x_\e-y_\e|)\\
&\leq&
\left( \gamma_m+\frac{(m-1)^{m-1}\gamma_m^m}{m^m ( (1-\sigma)\underline{c}- \omega(|x_\e-y_\e|))^{m-1}}\right)
|D\phi (x_\e)|^{m} +\gamma_1 |D\phi (x_\e)|+ (1-\sigma)A + \omega(|x_\e-y_\e|).
\end{eqnarray*}

Subtracting~(\ref{ineg-sous}) and~(\ref{ineg-sur}), letting $\e\to 0$ and using that $(2\e)^{-1}|x_\e-y_\e|^2\to 0$
it follows
\begin{eqnarray*}
&& {\cal P}^+_{\lambda,\Lambda}(D^2\phi (\hat{x}))
+ \left( \gamma_m+\frac{(m-1)^{m-1}\gamma_m^m}{m^m (1-\sigma)^{m-1}\underline{c}^{m-1}}\right)
|D\phi (\hat{x})|^{m} +\gamma_1 |D\phi (\hat{x})|\\
&& -\left( |u(\hat{x})|^{s-1}u(\hat{x}) -|v_\sigma(\hat{x})|^{s-1}v_\sigma(\hat{x})\right)
+ (\sigma - \sigma^s)|v(\hat{x})|^{s-1}v(\hat{x})
\geq
(1-\sigma) (f(\hat{x})-A)
\end{eqnarray*}
which proves that $w_\sigma$ is a viscosity subsolution of~(\ref{eqdiff})
with $\tilde{\gamma}$ given by~(\ref{tildgam}).


\end{document}